\documentclass[a4paper,12pt]{amsart}
\usepackage{amsfonts}
\usepackage[T1]{fontenc}
\usepackage{microtype}
\usepackage{setspace}
\usepackage{graphicx}
\usepackage[all]{xy}
\usepackage{amsmath}
\usepackage{amssymb}
\usepackage{booktabs}
\usepackage{array}
\usepackage{tabularx}
\usepackage{indentfirst}
\usepackage{cite}
    \renewcommand{\leq}{\leqslant}
    \renewcommand{\geq}{\geqslant}
\usepackage{amsthm}
\theoremstyle{plain}
\newtheorem{thm}{Theorem}[section]

\newtheorem{lem}[thm]{Lemma}
\newtheorem{prop}[thm]{Proposition}
\newtheorem{cor}[thm]{Corollary}

\theoremstyle{remark}
\newtheorem{oss}[thm]{Remark}
\newtheorem{ex}[thm]{Example}

\begin{document}

\title{The generalized lifting property of Bruhat intervals}
\author[Fabrizio Caselli]{Fabrizio Caselli$^\#$} \author [Paolo Sentinelli]{Paolo Sentinelli $^{*}$}

\address{$\#$ Dipartimento di matematica, Universit\`a di Bologna, Piazza di Porta San Donato 5, 40126 Bologna, Italy}
\address{${*}$Departamento de Matem\'aticas, Universidad de Chile,  Las Palmeras 3425, Nunoa, Santiago, Chile}

\email{$\#$ fabrizio.caselli@unibo.it}
\email{${*}$ paolosentinelli@gmail.com}

\thanks{${*}$ The second author was supported by MIUR grant FIRB-RBFR12RA9W-002 ``Perspectives in Lie theory''}
\date{}

\maketitle

\begin{abstract}In [E. Tsukerman and L. Williams, {\em Bruhat Interval Polytopes},
Advances in Mathematics, 285 (2015), 766-810] it is shown that every Bruhat interval of the symmetric group satisfies the so-called generalized lifting property. In this paper we show that a Coxeter group satisfies this property if and only if it is finite and simply-laced.
\end{abstract}

\section{Introduction}

The Kostant-Toda lattice is an integrable Hamiltonian system  which has been recently studied in detail by Kodama and Williams in \cite{KW}; in this paper particular attention is devoted to the asymptotic behaviour of the flow corresponding to an initial point associated with (a given point in) a cell $\mathcal R_{u,v}^+ $ of the totally non negative flag variety, where $[u,v]$ is a Bruhat interval in the symmetric group $S_n$.   When considering a natural multivariable generalisation of this problem, called the full Kostant-Toda hierarchy, they also proved that the moment polytope associated to the flow with  initial point corresponding  to $\mathcal R_{u,v}^+ $ is what they called a Bruhat interval polytope as it can be described as the convex hull in $\mathbb R^n $ of permutation vectors $(w(1),\ldots,w(n))$ as $w$ varies in the given Bruhat interval.

In \cite{TsukermanWilliams} Tsukerman and Williams studied some combinatorial aspects of a Bruhat interval polytope and in particular they found a dimension formula and proved that every face of a Bruhat interval polytope is itself a Bruhat interval polytope. The key result that they used in this study is what they called the generalized lifting property which generalizes the (standard) lifting property for symmetric groups and was surprisingly remained unnoticed so far. This generalized lifting property asserts that for every $u<v$ in the symmetric group and any ``minimal'' reflection (i.e. a transposition) $t$ such that $u<ut$ and $v>vt$, we have $u\lhd ut\leq v$ and $u\leq vt\lhd v$, where the symbol $\lhd$ denotes the covering relation in Bruhat order. The main fact about this property is that such minimal reflection always exists. The (standard) lifting property of a Bruhat interval is a classical feature of the theory of a Coxeter group $(W,S) $: it says that if $[u ,v] $ is a closed Bruhat interval in $W$ and $s\in S $ are such that $v> vs $ and $u<us $ then
 $u\lhd us\leq v $ and $u\leq vs\lhd v $, although it does not ensure that such simple reflection $s$ exists.  This property is well-known and is a basic tool in the combinatorics and geometry of Coxeter groups (see, e.g., \cite[Chapter 5]{humphreysCoxeter}, \cite[Chapter 2]{bjornerbrenti}, \cite{duCloux}) and has also found important applications in the combinatorics of Kazhdan-Lusztig polynomials (see, e.g., \cite{BCM, Del}).

The main target of this paper is to understand which Coxeter groups satisfy such generalized lifting property; our final result is that a Coxeter group satisfies this property if and only if it is finite and simply-laced. The proof uses in an essential way a geometric representation of a (finite) Coxeter group as a group generated by reflections in a (Euclidean) vector space. Several steps in the proof are done in complete generality, but there are some other peculiar results whose proofs still make use of the classification of finite Coxeter groups, and in particular we have to consider Weyl groups of type $D$ and $E$ separately. We think it would be very interesting to find a classification-free proof of our result.

The paper is organized as follows. In \S\ref{notation} we establish the notation, we introduce all the preliminary results in the theory of reflection and Coxeter groups that are needed, and we show some general results which link the generalized lifting property to the presentation of a Coxeter group as a reflection group.
In \S\ref{typeD} and in \S\ref{typeE} we develope in a similar and parallel way our theory for Weyl groups of type $D$ and $E$ respectively.  In \S\ref{affine} we consider affine Weyl groups and we prove some further basic facts that will allow us to complete the proof of our main result in \S\ref{sectionGLP}.
\section{Notation and preliminaries}\label{notation}

We begin by establishing some notation. $\mathbb{N}$ is the set of
non-negative integers and, if $n\in \mathbb{N}$,
$[n]:=\{1,2,...,n\}$; in particular $[0]=\varnothing$. We denote by $|X|$
the cardinality of a set $X$.

Next we recall some basic results in the theory of
Coxeter groups and reflection groups which
will be useful in the sequel. The reader can consult  the fundamental books \cite{bjornerbrenti, brown, humphreysCoxeter}
as comprehensive sources for this theory and in particular for the undefined notation, results stated without proof, and for further details. 

Let $(W,S)$ be a Coxeter system. If
$v,w\in W$ we define $\ell(v,w):=\ell(w)-\ell(v)$, where $\ell(z)$
is the length of the element $z\in W$ with respect to $S$. We let \begin{gather*}  D_L(w):=\{s\in S|\ell(sw)<\ell(w)\},
\\ D_R(w):=\{s\in S|\ell(ws)<\ell(w)\}.
\end{gather*}

The parabolic subgroup $W_J \subseteq W$ is the subgroup with $J\subseteq S$
as generator set. In particular $W_S=W$ and $W_\varnothing =
\{e\}$.

We consider on $W$ the Bruhat order $\leqslant$ (see, e.g., \cite[Chapter
2]{bjornerbrenti} or \cite[Chapter 5]{humphreysCoxeter}). With $[u,v]$ is denoted an interval in
$W$, i.e., if $v,w\in W$, $$[v,w]:=\{z\in W|v\leqslant
z\leqslant w\}.$$ The relation $v\vartriangleleft w$ means that $u < v$ and $\ell(u,v)=1$. 

We recall the following fundamental property of the Bruhat order, known as the
\emph{lifting property} (see \cite[Proposition 2.2.7]{bjornerbrenti}):
\begin{prop} \label{sollevamento} Let $v,w\in W$ be such that $v<w$ and $s\in D_R(w)\setminus D_R(v)$. Then $v\leqslant
ws$ and $vs\leqslant w$.
\end{prop}

A Coxeter system $(W,S)$ is called \emph{simply-laced} if $m(s,r)\leqslant 3$ for every $s,r\in S$, $m$ being its Coxeter matrix.

The set $T=\{wsw^{-1}|s\in S, w\in W\}$ is the set of reflections of a Coxeter system $(W,S)$.
For a Coxeter group  $W$ we let, for any $u,v\in W$,
\begin{eqnarray*} D(u) &:=& \{t \in T | ut<u\}, \\ A(u) &:=& \{t \in T | u<ut\}, \\ AD(u,v) &:=& A(u)\cap D(v).
\end{eqnarray*}

A finite reflection group  is a finite subgroup of $GL(V)$, where $V$ is a finite dimensional real vector space, which is generated by reflections, i.e. elements  of order 2 that fix a hyperplane pointwise.
Let $W$ be a finite reflection group and $T$ be the set of all reflections in $W$.
For each reflection $t\in T$ we denote by $H_t$ the hyperplane fixed by $t$ and call this the reflecting hyperplane associated to the reflection $t$. Moreover, for every reflection $t\in T$ one can choose a non-zero vector $\alpha_t\in V $ such that
\begin{itemize}
 \item $t(\alpha_t)=-\alpha_t$ for all $t\in T$;
 \item the set  $\Phi=\{\pm\alpha_t:\, t\in T\}$ is invariant under the action of $W$.
\end{itemize}
The set $\Phi$ is called the set of roots and we assume without lack of generality that $\Phi$ spans $V$. The connected components of the complement of the union of all reflecting hyperplanes are called chambers and we recall that the action of $W$ on the set of chambers is simply transitive. If we fix a chamber $\mathcal C_0$ (that will be called fundamental chamber) we let $H_t^+$ be the halfspace determined by $H_t$ which contains $\mathcal C_0$ and we let
\[\Phi^+=\Phi^+(\mathcal C_0)=\bigcup_{t\in T}\{\pm \alpha_t\} \cap H_t^+.\] Then there exists a unique set of roots $\Delta \subset \Phi^+$, called base, such that
\begin{itemize}
\item $\Delta$ is a linear basis of $V$;
\item every root in $\Phi^+$ can be expressed as a linear combination with non-negative coefficients of the elements in $\Delta$.
\end{itemize}
Changing $\alpha_t$ with $-\alpha_t$ if necessary we can assume that $\Phi^+=\{\alpha_t,\, t\in T\}$. If we let $S\subset T$ be the indexing set of $\Delta$, i.e.
$$\Delta=\{\alpha_s:\,s\in S\}$$
we have that $(W,S)$ is a finite Coxeter system whose set of reflections is exactly $T$, and every finite Coxeter system arises in this way as a finite reflection group. In the sequel we will always assume that a finite Coxeter group comes equipped with the structure of a finite reflection group as above.

The set of reflections $T$ is partially ordered by letting, for all $r,t\in T$, $r\prec t$ if $\alpha_t-\alpha_r$ is still a linear combination of the roots $\alpha_s$, $s\in S$, with non-negative coefficients.

The length function of an element $w\in W$ and the set $D(w)$ have the following geometric interpretation which will be fundamental in our study. The length of an element $w\in W$ equals the number of reflecting hyperplanes which separate the fundamental chamber $\mathcal C_0$ from the chamber $\mathcal C_w: =w^{-1}(\mathcal C_0)$. A direct consequence is that the set $D(w)$ equals the set of reflections $t\in T$ such that the associated reflecting hyperplane separates $\mathcal C_w$ from $\mathcal C_0$, and in particular we have

\begin{equation} \label{lunghezzaw}
  \ell(w)=|D(w)|.
\end{equation}
\begin{oss} \label{remarkAD}
 An immediate consequence of \eqref{lunghezzaw} is that, if $\ell(u)<\ell(v)$ and in particular if $u<v$, then $|AD(u,v)|\geqslant \ell(u,v)>0$,
since $|AD(u,v)|=\ell(u,v)+|D(u)\setminus D(v)|$. In particular, if $|AD(u,v)|=1$ we have $\ell(u,v)=1$.
\end{oss}
We will be interested also in the action of reflections on halfspaces determined by the reflecting hyperplanes.  It is clear that for any reflection $t$ we have $t(H_r)=H_{r^t}$ where $r^t:=trt$ and therefore we have that either $t(H_r^+)=H_{r^t}^+$ or $t(H_r^+)=H_{r^t}^-$.

The following result is crucial in our work.
\begin{prop}\label{crucialprop}
 Let $r,t\in T$. The following are equivalent
 \begin{enumerate}
 \item $t(H_r^+)=H_{r^t}^-$;
  \item $r\in D(t)$;
  \item for some $u\in W$ we have $r\in D(u)$ if and only if $r^t\in A(ut)$.
  \item for every $v\in W$ we have $r\in D(v)$ if and only if $r^t\in A(vt)$.
 \end{enumerate}
\end{prop}
\begin{proof}
(1) implies (2).  Since $\mathcal C_r\subset H_r^-$ we have $\mathcal C_{rt}=t(\mathcal C_r)\subset H_{r^t}^+$. Therefore $r^t\in A(rt)$ i.e. $t>rt$. Since the map $w\mapsto w^{-1}$ is an automorphism of $W$ as a poset, we also have $t>tr$ i.e. $r\in D(t)$.

(2) implies (3).  Letting $u=t$ we clearly have $r\in D(u)$ and $r^t\in A(e)$.

(3) implies (1).  We have either $\mathcal C_u\in H_r^-$ and $\mathcal C_{ut}\in H_{r^t}^+$ or $\mathcal C_u\in H_r^+$ and $\mathcal C_{ut}\in H_{r^t}^-$. As $\mathcal C_{ut}=t (\mathcal C_u)$ it follows that $t(H_r^-)=H_{r^t}^+$.

(1) implies (4). Let $v\in W$. We have $r\in D(v)$ if and only if $\mathcal C_v\in H_r^-$ and $r^t\in A(vt)$ if and only if $\mathcal C_{vt}=t(\mathcal C_v)\in H_{r^t}^+$ and since $t(H_r^-)=H_{r^t}^+$ the two conditions are equivalent.

(4) implies (2) Take $v=t$ and the result follows.
\end{proof}

\begin{cor}\label{commrefl}
 Let $r,t\in T$, $r\neq t$, be such that $rt=tr$. Then $r \in A(t)$.
\end{cor}
\begin{proof}
By contradiction, if $r\in D(t)$, we have $t(H_r^+)=H_r^-$, by Proposition \ref{crucialprop}. This implies $H_t\subseteq H_r$ and so $H_r=H_t$; moreover $r(\alpha_t)$ is still an eigenvector for $t$ of eigenvalue $-1$, since $r$ and $t$ commute. Therefore $r(\alpha_t)=c\alpha_t$ and since $r$ is reflection we necessarily have $c=-1$ and hence $r=t$.
\end{proof}

If $t\in T$, we have the following partition of the set $T$:
$$T= \{t\} \sqcup (D(t)\setminus \{t\}) \sqcup A(t),$$
and an involution $T \rightarrow T$ defined as the conjugation
by $t$, i.e.
$$r^t:=trt,$$ for all $r\in T$.

\begin{lem} \label{lemmainvarianza} The sets $\{t\}$, $D(t)\setminus \{t\}$ and $A(t)$ are invariant under conjugation by $t$, for all $t\in T$.
\end{lem}
\begin{proof} Clearly $\{t\}$ is invariant. Let $r\in D(t)\setminus \{t\}$, i.e. $e<rt<t$ or, equivalently $e<tr=(rt)^{-1}< t$, since the inversion is an automorphism of the poset $(W,\leqslant)$; then $e<r^tt=tr< t$. It follows that $A(t)$ is invariant too.
\end{proof}





The following theorem generalizes, in the symmetric group, the lifting property stated in Proposition \ref{sollevamento}; this result will be extended to the case of finite simply-laced Coxeter groups and this is the main motivation of the paper.
\begin{thm}{\cite[Theorem 3.3]{TsukermanWilliams}} \label{GLPsymmetric}
 Let $[u,v]$, $u<v$, be a Bruhat interval in the symmetric group $A_n$ and $t$ be a minimal element in $AD(u,v)$ with respect to $\prec$. Then $u\vartriangleleft ut \leqslant v$ and $u\leqslant vt \vartriangleleft v$.
\end{thm}

We end this section recalling some known facts about the $R$-polynomials; see \cite[Chapter 5]{bjornerbrenti}, \cite[Chapter 7]{humphreysCoxeter} and the references there for more information. We just recall here that for any Coxeter group $W$ there exists a unique family of
polynomials $\{R_{u,v}\}_{u,v\in W}\subseteq
\mathbb{Z}[q]$ such that, for all $u,v\in W$:
\begin{enumerate}
  \item $R_{u,v}=0 ~~\mbox{if $u \nleqslant v$}$;
  \item $R_{v,v}=1$;
  \item  if $u < v$ and $s\in D_R(v)$ then
  \begin{equation}\label{ricorsioneRpolinomi} R_{u,v}=\begin{cases}
  R_{us,vs}, &\mbox{if $s\in D_R(u)$,} \\ qR_{us,vs}+(q-1)R_{u,vs}, &\mbox{otherwise.} \end{cases}
  \end{equation}
\end{enumerate} These polynomials are called $R$-\emph{polynomials}. In the hypothesis of Theorem \ref{GLPsymmetric}, the simple reflection in the recursion \eqref{ricorsioneRpolinomi} can be replaced by any minimal reflection in $AD(u,v)$.

\begin{prop} {\cite[Proposition 5.3]{TsukermanWilliams}} \label{propRsimm} Let $[u,v]$, $u<v$, be a Bruhat interval in the symmetric group $A_n$ and $t$ be a minimal element in $AD(u,v)$. Then $$R_{u,v}=qR_{ut,vt}+(q-1)R_{u,vt}.$$
\end{prop} As shown in \cite[Section 5]{TsukermanWilliams}, if $t\in AD(u,v)$ is any reflection such that $u\vartriangleleft ut \leqslant v$ and $u\leqslant vt \vartriangleleft v$, the result of the proposition could be not true.

\section{Type D}\label{typeD}

As explained in the introduction, some properties needed in the proof of our main result are proved using the classification of finite Coxeter groups. In this section we examine in detail the case of Weyl groups of type $D_n$ and show all the preliminary results that will be needed to prove that the generalized lifting property holds for every interval of $D_n$.

For any integer $n\geq 1$ the group $D_n$ can be identified with the set of signed permutations with an even number of negative entries
(see \cite[Chapter 8]{bjornerbrenti} for more details). If $w\in D_n$ we write $w=[w(1),\ldots,w(n)]$ and call this the window notation of $w$.
A set of Coxeter generators is $S = \{s_{i,i+1}|1\leqslant i \leqslant n-1\} \cup \{t_{1,2}\}$ and the set of reflections of
the Coxeter system $(D_n,S)$ is
\begin{equation*}
  T=\{s_{i,j}|1\leqslant i<j \leqslant n\} \cup \{t_{i,j}|1\leqslant i<j \leqslant n\},
\end{equation*} where, as elements of $D_n$ in the window notation,
\begin{eqnarray}
\label{riflessioniDs}  s_{i,j} &=& [1,2,...,i-1,j,i+1,...,j-1,i,j+1,...,n], \\
\label{riflessioniDt}  t_{i,j} &=& [1,2,...,i-1,-j,i+1,...,j-1,-i,j+1,...,n],
\end{eqnarray} for all $1\leqslant i<j \leqslant n$.
For notational convenience, if $i>j$ we also let $s_{i,j}:=s_{j,i}$ and $t_{i,j}:=t_{j,i}$.
For a proof of the next proposition see \cite[Proposition 8.2.1]{bjornerbrenti}.
\begin{prop} \label{lunghezzaD} Let $w\in D_n$ and $i<j$; then $s_{i,j}\in D(w)$ if and only if $w(i)>w(j)$ and $t_{i,j}\in D(w)$ if and only if $w(i)+w(j)<0$. In particular
\[\ell(w)=|\{(i,j)\in [n]\times [n]:i<j,w(i)>w(j)\}|+|\{(i,j)\in [n]\times [n]:i<j,w(i)+w(j)<0\}|.\]
\end{prop}
For example, $\ell(s_{i,j})=2(j-i)-1$ and $\ell(t_{i,j})=2(j+i-3)+1$.

In a classical realization of $D_n $ as a finite reflection group in $GL(\mathbb{R}^n)$, if we denote by $\alpha_{i,j}$ the positive root corresponding  to the reflection $s_{i,j}$ and by $\beta_{i,j}$ the positive root corresponding to the reflection $t_{i,j}$, we have   $\alpha_{i,j}=e_i-e_{j}$ and  $\beta_{i,j}=-e_i-e_j$ for all $1\leqslant i<j \leqslant n$, where
 $\{e_1,...,e_n\}$ is the canonical basis of $\mathbb{R}^n$. We have
\begin{eqnarray}
\label{positiverootsDs}  \alpha_{i,j} &=& \sum \limits_{k=i}^{j-1}\alpha_{k,k+1}, \\
\label{positiverootsDt}  \beta_{i,j} &=& \beta_{1,2} + \sum \limits_{k=1}^{i-1}\alpha_{k,k+1}+\sum \limits_{k=2}^{j-1}\alpha_{k,k+1},
\end{eqnarray} for all $1\leqslant i<j \leqslant n$.
By \eqref{positiverootsDs} and \eqref{positiverootsDt} we obtain the following description of the poset $(T,\preceq)$.
\begin{prop} \label{posetTD}
  For the Coxeter system $(D_n,S)$ the relations in the poset $(T,\preceq)$ are the following:
  \begin{enumerate}
 \item $s_{i,j}\preceq s_{h,k}$ if and only if $h\leqslant i$ and $j\leqslant k$;
 \item $s_{i,j}\prec t_{h,k}$ if and only if $(i,h)\neq (1,1)$ and $j\leqslant k$;
 \item $t_{i,j}\preceq t_{h,k}$ if and only if $i\leqslant h$ and $j\leqslant k$;
 \item $t_{i,j}\not \prec s_{h,k}$,
\end{enumerate}
for all $1\leqslant i<j \leqslant n$, $1\leqslant h<k \leqslant n$.
\end{prop}

For any set $I=\{i,j\}$ and $x\in I$, define ${\overline x}^I=i$ if $x=j$ and ${\overline x}^I=j$ otherwise. The proof of the following result is a simple verification and is therefore omitted.

\begin{prop} \label{twistDcor2} Let $r\in T$ and $i,j\in [n]$, $i \neq j$, $I=\{i,j\}$. Then
\begin{equation*}
 r^{s_{i,j}}=\begin{cases}
              s_{h,{\overline x}^I}& \textrm{if }r=s_{h,x},\,x\in I,\,h\notin I\\
              t_{h,{\overline x}^I}& \textrm{if }r=t_{h,x},\,x\in I,\,h\notin I\\
              r& \textrm{otherwise.}
             \end{cases}
\end{equation*}
and
\begin{equation*}
 r^{t_{i,j}}=\begin{cases}
              t_{h,{\overline x}^I}& \textrm{if }r=s_{h,x},\,x\in I,\,h\notin I\\
              s_{h,{\overline x}^I}& \textrm{if }r=t_{h,x},\,x\in I,\,h\notin I\\
              r& \textrm{otherwise.}
             \end{cases}
\end{equation*}
%
\end{prop}
If $t\in T$, the set $D(t)$ can be completely described.
\begin{prop} \label{twistD}
  Let $1\leqslant i<j \leqslant n$ and $I:=\{i,j\}$; then
  \begin{equation*}
    D(t)\setminus \{t\}=\begin{cases}
           \{s_{h,x}|x\in I, i<h<j\}, & \mbox{if $t=s_{i,j}$}, \\
           \{s_{h,x}, t_{h,{\overline x}^I}|x\in I, h\not \in I,h<x\}, & \mbox{if $t=t_{i,j}$}.
         \end{cases}
  \end{equation*}
\end{prop}
\begin{proof}
We prove the result for $t_{i,j}$, the case $s_{i,j}$ being similar and simpler.  By Corollary \ref{commrefl}, Lemma \ref{lemmainvarianza} and Proposition \ref{twistDcor2} it is enough to show that for $x\in \{i,j\}$ and $h\notin \{i,j\}$ we have $s_{h,x}\in D(t_{i,j})$ if and only if $h<x$. But this follows immediately from Proposition \ref{lunghezzaD}.
\end{proof}

An immediate consequence of the description of the poset $(T,\preceq)$ in Proposition \ref{posetTD} and Proposition \ref{twistD} is the following observation.
\begin{cor} \label{twistDcor} Let $t\in T$; then $r\prec t$ for all $r\in D(t)\setminus \{t\}$.
\end{cor}

The following is a key result used in the proof of the generalized lifting property in \S\ref{sectionGLP}.
\begin{thm} \label{tt*Dn}
  Let $u,v\in D_n$ be such that $AD(u,v)\neq \emptyset$, $t$  be a minimal element in $AD(u,v)$ and $r\in D(t)\setminus \{t\}$; then
  $$r\in A(u) \Leftrightarrow r\in A(v) \Leftrightarrow r^t \in D(v) \Leftrightarrow r^t \in D(u). $$
\end{thm}
\begin{proof} First note that $r\in A(u) \Leftrightarrow r\in A(v)$ and $r^t\in D(u) \Leftrightarrow r^t \in D(v)$,
by the minimality of $t$, Lemma \ref{lemmainvarianza} and  Corollary \ref{twistDcor}.
 To prove the other implications there are two cases to consider.
\begin{enumerate}
  \item $t=s_{i,j}$: in this case we have $r=s_{h,x}$ for some $x\in \{i,j\}$ and $i<h<j$ by Proposition \ref{twistD}, and $r^t=s_{h,{\overline x}^I}$ by Corollary \ref{twistDcor2}.
If $r\in A(v)$, i.e. $v_j<v_x<v_h$ if $x=i$
and $v_h<v_x<v_i$ if $x=j$, then $s_{h,{\overline x}^I}\in D(v)$. If
$s_{h,{\overline x}^I}\in D(v)$ then $v_h>v_j>v_i$ if $x=i$ and $v_j>v_i>v_h$ if $x=j$, therefore $s_{h,x}\in A(v)$.
  \item $t=t_{i,j}$: in this case $u_i+u_j>0$, $v_i+v_j<0$ and we can assume $r=s_{h,x}$, where $x\in \{i,j\}$, $h\not \in \{i,j\}$ and $h<x$ by Propositions \ref{twistD} and Corollary \ref{twistDcor2}, since the statement for a reflection $r\in D(t)$ is equivalent to the same statement for $r^t$. By Corollary \ref{twistDcor2} we also have $r^t=t_{h,{\overline x}^I}$. Assume $r\in A(v)$; therefore $v_h<v_x$. So $v_h+v_{{\overline x}^I}<v_x+v_{{\overline x}^I}<0$, which implies $t_{h,{\overline x}^I}\in D(v)$. Finally, if $t_{h,{\overline x}^I}\in D(u)$ we have
  $u_h+u_{{\overline x}^I}<0$; since $u_i+u_j=u_x+u_{{\overline x}^I}>0$ we find $u_h<u_x$, i.e. $s_{h,x}\in A(u)$.
\end{enumerate}
\end{proof}

This last proposition of this section will be crucial for the inductive step in the proof of Theorem \ref{GLPfinite}.
\begin{prop}\label{goodsDn}
 Let $u,v\in D_n$ and $t$ be a minimal element in $AD(u,v)$. If $t\notin S$ then there exists $s\in S$ such that
 \begin{enumerate}
  \item $s\in D(t)$;
  \item $s^t$ is minimal in $AD(us,vs)$,
  \item $s^t=t^s$.
 \end{enumerate} Moreover $s\in D_R(v)$ if and only if $s\in D_R(u)$.
\end{prop}
\begin{proof}
Let $1\leqslant i<j\leqslant n$. If $t=s_{i,j}$ choose $s=s_{j-1,j}$; then $s\in D(t)$ and $tst=s_{i,j-1}=sts$.
Since $(us)_i=u_i$, $(us)_{j-1}=u_j$ we have $s_{i,j-1}\in A(us)$ and similarly we can show that $s_{i,j-1}\in D(vs)$. Assume there exists $r\in AD(u,v)$ such that $r\prec s_{i,j-1}$. By the description of the partial order $\prec $ we have that $r=s_{h,k}$ with $i\leq h<k\leq j-1$. If $k<j-1$ then $r\in AD(u,v)$ contradicting the minimality of $t$. If $k=j-1$ then $rs=s_{h,j}\in AD(u,v)$ contradicting again the minimality of $t$.

  Now let $t=t_{i,j}$. If $i>1$ then $s=s_{i-1,i} \in D(t)$ and $tst=t_{i-1,j}=sts$.
  If $i=1$ then $j>2$ and we can let $s=s_{j-1,j}$. Then $s\in D(t)$ and $tst=t_{i,j-1}=sts $.
  The fact that $s^t$ is a minimal element in $AD(us,vs)$ can be proved as in the former case and is left to the reader.
  The last statement follows directly by Theorem \ref{tt*Dn}.
 \end{proof}











\section{The group $E_8$}\label{typeE}

This section is devoted to the study of the Weyl group (of type) $E_8$ and in particular we are going to obtain parallel results to those obtained in \S\ref{typeD} for Weyl groups of type $D_n$. As many results are rather technical we will only sketch some of the proofs. We collect here a description of a possible realization of the Weyl group $E_8$ as a finite reflection group that is mainly developed in \cite{eriksson}.
First we fix some notation. For any $H\subset [8]$ and any vector $x=(x_1,\ldots,x_8)\in \mathbb R^8$ we let
\[
 \Sigma_H(x):= \sum_{h\in H}x_h
\]
and $\Sigma(x):=\Sigma_{[8]}(x)$.
Next we introduce a set of reflections on $\mathbb R^8$.
\begin{itemize}
 \item For all $I\subset [8]$, $|I|=2$, $I=\{i,j\}$, we let $s_{I}\in GL(\mathbb R^8)$ be given by
 \[
  s_{I}(x)_k=\begin{cases}
                  x_k&\textrm{if $k\neq i,j$}\\ x_i&\textrm{if $k= j$}\\x_j&\textrm{if $k= i$},
                 \end{cases}
\]
for all $x=(x_1,\ldots,x_8)\in \mathbb{R}^8$.
The endomorphism $s_{I}$ is a reflection as it fixes pointwise the hyperplane of equation $x_i-x_j=0$ and sends the vector $\alpha_{I}:=e_j-e_i$ (where we assume $i<j$ and we denote by $e_1,\ldots,e_8$ the canonical basis elements of $\mathbb R^8$) to its negative;
\item for all $H\subset[8]$ such that $|H|=3,6$ we let ${s}_H \in GL(\mathbb R^8)$ be given by
$${s}_H(x)_i:=\begin{cases}
            x_i,&\textrm{if }i\notin H\\ x_i+\frac{|H|}{9}\Sigma(x)-\Sigma_H(x),&\textrm{if }i\in H.
           \end{cases}$$
The endomorphism ${s}_H$ is a reflection as it fixes pointwise the hyperplane of equation $\frac{|H|}{9}\Sigma(x)-\Sigma_H(x)=0$ and sends the vector ${\alpha}_H=\sum_{h\in H}e_h$ to its negative.
\item for all $L\subset [8]$ with $|L|=1$, $L=\{l\}$,  we let ${s}_L  \in GL(\mathbb R^8)$ be given by
$$ {s}_L(x)_i =\begin{cases}
            x_i-x_l,&\textrm{if }i\neq l,\\ -x_l,&\textrm{if }i=l.
           \end{cases}$$
           The endomorphism ${s}_L$ is also a reflection as it fixes pointwise the hyperplane of equation $x_l=0$ and sends the vector
           ${\alpha}_L=e_l+\sum_{i=1}^8 e_i$ to its negative.
\end{itemize}
If we let
$$T=\{s_H:\,   H\subset [8] \textrm{ and }|H|=1,2,3,6\}.$$
one can check that the set
\[
 \Phi:=\{\pm \alpha_{H}:\, H\subset [8] \textrm{ and }|H|=1,2,3,6\}
\]
is invariant under the action of all reflections in $T$ and therefore the group $W$ generated by $T$ is a finite reflection group. If $H=\{h_1,\ldots,h_r\}$ with $r=|H|=1,2,3,6$ we also write
$s_{h_1,\ldots,h_r}$ and $\alpha_{h_1,\ldots,h_r}$ instead of $s_{\{h_1,\ldots,h_r\}}$ and $\alpha_{\{h_1,\ldots,h_r\}}$ respectively.

If we take the following as the fundamental chamber
\[
 \mathcal C_0:=\{x\in \mathbb R^8:\, x_1<x_2<\cdots <x_8,\, \frac{1}{3}\Sigma(x)-\Sigma_{\{1,2,3\}}(x)<0\}.
\]
we have that the corresponding base of $\Phi$ is given by
\[
 \Delta:=\{\alpha_{1,2},\ldots,\alpha_{7,8},{\alpha}_{1,2,3}\}
\]
and so the corresponding set of simple reflections is
\[
 S:=\{s_{1,2},\ldots,s_{7,8},{s}_{1,2,3}\}
\]
and it is now a straightforward check that $(W,S)$ is a Coxeter system of type $E_8$ and we simply denote it by $E_8$. The set of positive roots is given by
all ${\alpha}_H$, $|H|=1,2,3,6$ and we have the following decompositions of the positive roots:
\begin{enumerate}
 \item $\alpha_{i,j}=\sum\limits_{k=i}^{j-1}\alpha_{i,i+1}$, for all $i<j$;
 \item ${\alpha}_J={\alpha}_{1,2,3}+\alpha_{1,j_1}+\alpha_{2,j_2}+\alpha_{3,j_3}$, for all $J=\{j_1,j_2,j_3\}$ such that $ 1\leqslant j_1<j_2<j_3\leqslant 8$;
 \item ${\alpha}_K=2{\alpha}_{1,2,3}+\alpha_{1,k_1}+\alpha_{2,k_2}+\alpha_{3,k_3}+\alpha_{1,k_4}+\alpha_{2,k_5}+\alpha_{3,k_6}$, for all
 $K=\{k_1,k_2,\ldots,k_6\}$ such that $1\leqslant k_1<k_2<\cdots< k_6 \leqslant 8$.
 \item ${\alpha}_l=3{\alpha}_{1,2,3}+\alpha_{2,4}+\alpha_{3,5}+\alpha_{1,6}+\alpha_{2,7}+\alpha_{3,8}+\alpha_{1,l}$, for  all $l\in[8]$,
\end{enumerate}
where we let $\alpha_{i,i}:=0$ for all $i\in[8]$.
We are therefore in a position to describe explicitly the partial order $\prec$ on $T$.
Let $i<j$, $h<k$, $J=\{j_1,j_2,j_3\}$ with $j_1<j_2<j_3$, $J'=\{j_1',j_2',j_3'\}$ with $j_1'<j_2'<j_3'$, $K=\{k_1,k_2,\ldots,k_6\}$ with $k_1<k_2<\cdots <k_6$, $K'=\{k_1',k_2',\ldots,k_6'\}$ with $k_1'<k_2'<\cdots <k_6'$, $l,l'\in [8]$. We have

\begin{itemize}
\item $s_{i,j}\prec s_{h,k}$ if $h\leq i<j\leq k$;
 \item $s_{i,j}\prec {s}_J$ if $j\leq j_3$, $(i,j_1)\neq (1,1)$, $(i,j_2)\neq (2,2)$;
 \item $s_{i,j}\prec {s}_K$ if $j\leq k_6$;
 \item $s_{i,j}\prec {s}_l$;
 \item ${s}_J\prec {s}_{J'}$ if $j_1\leq j_1'$, $j_2\leq j_2'$, $j_3\leq j_3'$;
 \item ${s}_J\prec {s}_K$ if $j_1\leq k_4$, $j_2\leq k_5$, $j_3\leq k_6$;
 \item ${s}_J\prec {s}_l$;
 \item ${s}_K\prec {s}_{K'}$ if $k_1\leq k_1',\ldots,k_6\leq k_6'$;
 \item ${s}_K\prec {s}_l$ if $k_1\leq l$;
 \item ${s}_l\prec {s}_{l'}$ if $l\leq l'$.
\end{itemize}

As a consequence of our choice of $\mathcal C_0$ we also have the following description of the halfspaces $H_t^+$, as $t\in T$.
\begin{itemize}
 \item $H_{s_{i,j}}^+:\, x_i-x_j<0$, for $i<j$;
 \item $H_{{s}_H}^+:\, \frac{|H|}{9}\Sigma(x)-\Sigma_H(x)<0$, for all $H\subset[8]$ such that $|H|=3,6$;
  \item $H_{{s}_l}^+:\, x_l>0$, for all $l\in[8]$.
\end{itemize}

Therefore we can also explicitly determine the sets $D(w)$ for all $w\in E_8$.
\begin{prop}Let $w\in E_8$ and $x\in \mathcal C_w=w^{-1}(\mathcal C_0)$. Then
\begin{align*}s_{i,j}\in D(w) & \Leftrightarrow x_i-x_j>0 \textrm{, for all $i<j$}; \\
{s}_H\in D(w) &\Leftrightarrow  \frac{|H|}{9}\Sigma(x)-\Sigma_H(x)>0 \textrm{, for all $H\subset [8],\, |H|=3,6$};\\
{s}_l\in D(w) &\Leftrightarrow  x_l<0 \textrm{, for all $l\in [8]$}.
\end{align*}
In particular we have
\[
 \ell(w)=|\{i<j: x_i>x_j\}|+|\{H:\, |H|=3,6,\, \frac{|H|}{9}\Sigma(x)-\Sigma_H(x)>0\}|+|\{l:\, x_l<0\}|.
\]
\end{prop}

Our next target is to give,  for any reflection $t\in T$,  an explicit description of the map $r\mapsto r^t$ and of the set $D(t)$. We state without proof the following two easy lemmas for future reference.
\begin{lem}\label{forti}
 Let $l\in [8]$ and $H\subset [8]$, $|H|=3,6$. Then
 \begin{itemize}
  \item $\Sigma({s}_l(x))=\Sigma(x)-9x_l$;
  \item $\Sigma_H({s}_l(x))=\Sigma_H(x)-\big(|H|+\chi(l\in H)\big)x_l,$
 \end{itemize}
 where
 \[
 \chi(l\in H)=\begin{cases} 1& \textrm{if }l\in H\\0&\textrm{otherwise.}\end{cases}
 \]
\end{lem}

 \begin{lem}\label{fortH}
 Let $H,H'\subset[8]$ with $|H|,|H'|=3,6$. Then
\begin{itemize}
  \item $\Sigma({s}_H(x))=(|H|-1)\Sigma(x)-|H|\Sigma_H(x)$;
  \item $\Sigma_{H'}({s}_H(x))=\Sigma_{H'}(x)+|H\cap {H'}|\Big(\frac{|H|}{9}\Sigma(x)-\Sigma_H(x)\Big)$
 \end{itemize}
\end{lem}

For $H,H'\subset [8]$ we denote by $H\triangle H':= (H\cup H')\setminus (H\cap H')$ the symmetric difference of $H$ and $H'$, and by $\bar H=[8]\setminus H$ the complement set of $H$.
\begin{prop}\label{DtL}
Let $l \in [8]$, $L=\{l\}$ and $H\subset [8]$, $|H|=1,2,3,6$. Then
\[
s_H^{{s}_L}=\begin{cases}
       s_{H}  & \textrm{if either $|H\cap L|=0$ and $|H|\neq 1$ or $H=L$},\\
       s_{H\triangle L}  & \textrm{if $|H\cap L|=0$ and $|H|=1$ or $|H\cap L|=1$ and $|H|=2$},\\
       s_{\bar H\triangle  L}  & \textrm{if $|H\cap L|=1$ and $|H|=3,6$},
        \end{cases}
\]
and
\[
 D({s}_l)=\{s_{h,l},s_h,\,h <  l \}\cup \{s_H:\, |H|=3,6 \textrm{ and }l\in H\}\cup \{s_l\}.
\]

\end{prop}
\begin{proof}
 The description of the map $r\mapsto r^{s_{l}}$ is a straightforward verification. We prove it in one significant case only, leaving the details of all the other cases.
 We consider the case $s_J$ with $|J|=3$ and $L\subset J$. We have to prove that $s_J^{s_L}=s_{\bar J\triangle L}$ and for this it is enough to show that $s_L(H_{s_J})\subset H_{s_{\bar H\triangle L}}$, where we recall that we denote by $H_r$ the reflecting hyperplane associated to a reflection $r$. So let $x\in H_{s_J}$, i.e. such that $\frac{1}{3}\Sigma(x)-\Sigma_J(x)=0$. We have, by Lemma \ref{forti}, $\Sigma(s_L(x))=\Sigma(x)-9x_l$ and $\Sigma_{\bar J \triangle L}(s_L(x))=\Sigma_{\bar J\triangle L}(x)-7x_l=(\Sigma(x)-\Sigma_J(x)+x_l)-7x_l$. Therefore
 \[
  \frac{2}{3}\Sigma(s_L(x)-\Sigma_{\bar J\triangle L}(s_L(x))=0,
 \]
 i.e. $s_L(x)\in H_{s_{\bar H\triangle L}}$.

 We therefore proceed with the study of $D(s_{l})$. We let $x\in \mathcal C_0$.
 Let $h\in[8]$, $h\neq {l}$. We have $s_h\in D(s_{l})$ if and only if $(s_{l}(x))_h=x_h-x_{l}<0$ and this is satisfied if and only if $h<{l}$. Therefore, for $h<{l}$, $s_h\in D(t_{l})$. Since $s_h^{s_{l}}=s_{h,{l}}$ we also have $s_{h,{l}}\in D(s_{l})$. Similarly, $s_h$ and $s_{{l},h}$ are not elements in $D(s_{l})$ if $h>{l}$.

 If $|{H}|=2,3,6$ and ${l}\notin {H}$ we have that $s_{H}$ and $s_{l}$ commute and so $s_{H} \notin D(s_{l})$ by Corollary \ref{commrefl}.

 If ${l}\in J\subset [8]$, with $|J|=3$ we have ${s}_J\in D({s}_{l})$ if and only if $\frac{1}{3}\Sigma({s}_{l} x)-\Sigma_J({s}_{l}x)>0$. By Lemma \ref{forti} we have
   \begin{align*}
    \frac{1}{3}\Sigma({s}_{l} x)-\Sigma_J({s}_{l}x)&=\frac{1}{3}\Sigma(x)-3x_{l}-\Sigma_J(x)+4x_{l}\\
    &=\frac{1}{3}\Sigma(x)-\Sigma_J(x)+x_{l}\\
    &= -\frac{2}{3}\Sigma(x)+\Sigma_{\bar J \triangle L}(x)
   \end{align*}
   where we have used the simple observation $\Sigma_J(x)+\Sigma_{\bar J \triangle L}(x)=\Sigma(x)+x_{l}$. The result follows since $x\in \mathcal C_0$.

   If ${l}\in J\subset [8]$, with $|J|=6$ the result follows since $s_J^{t_{l}}=s_{\bar J \triangle L}$ and $|\bar J \triangle L|=3$.

\end{proof}

The proofs of the following three results, Propositions \ref{DtK}, \ref {DtJ} and \ref{DsI} are very similar to the proof of Proposition \ref{DtL} and hence they are omitted.
\begin{prop}\label{DtK}
Let $H,K\subset [8]$, $|K|=6$, $|H|=1,2,3,6$. Then
\[
{s_{\hspace{-0.5mm}\scriptscriptstyle H}}^{\hspace{-1.5mm}{s}_{\hspace{-0.3mm}K}}=\begin{cases}
         s_H & \textrm{if $ |H\cap K|=0,2,4,6$};\\
         s_{H\triangle K}& \textrm{if $|H\cap K|=3,5$, or $|H\cap K|=1$ and $|H|=2$};\\
         s_{\bar H\triangle K} & \textrm{if $|H\cap K|=1$ and $|H|=1,3$};
        \end{cases}
\]
and
\[
 D({s}_K)=\{ s_I,s_{K\triangle I}:\, I=\{i,j\},\, i<j,\, i\notin K,\, j\in K \}\cup \{ s_J,\, |J|=3,\, J\subset K\}\cup\{s_K\}.
\]
\end{prop}

\begin{prop}\label{DtJ}
Let $J, H\subset [8]$, $|J|=3$, $|H|=1,2,3,6$. Then
\[
{s_{\hspace{-0.5mm}\scriptscriptstyle H}}^{\hspace{-1.5mm}{s}_{\hspace{-0.3mm}J}}=\begin{cases}
         s_H& \textrm{if $(|H\cap J|,|H|)=(0,1),(0,2),(2,6),(1,3),(3,3)$ }\\
         s_{H\triangle J}& \textrm{if $(|H\cap J|,|H|)=(2,3),(1,2),(3,6),(0,3)$}\\
        s_{\bar H\triangle J} & \textrm{if $(|H\cap J|,|H|)=(1,1),(1,6)$}\\
                \end{cases}
\]
and
\[
 D({s}_J)=\{s_I,\,s_{I\triangle J:\, I=\{h,j\}},\, j\in I,\, h\notin I,\, h<j \}\cup \{ s_J\}.
\]
\end{prop}

\begin{prop}\label{DsI}
Let $I, H\subset [8]$, $I=\{i,j\}$, with $i<j$ and $|H|=1,2,3,6$. Then
\[
{s_{\hspace{-0.5mm}\scriptscriptstyle H}}^{\hspace{-1.5mm}{s}_{\hspace{-0.3mm}I}}=\begin{cases}
         s_H& \textrm{if $|H\cap I|=0,2$}\\
         s_{H\triangle I }& \textrm{if $|H\cap I|=1 $}\\
         \end{cases}
\]
and
\[
 D({s}_I)=\{ s_{i,h},\, s_{h,j}:\, i<h<j\}\cup \{s_{i,j} \}.
\]
\end{prop}

An immediate consequence of Propositions \ref{DtL}, \ref{DtK}, \ref{DtJ} and \ref{DsI} is the following.
\begin{cor}\label{rprect}
Let $r,t$ be reflections such that $r\in D(t)$. Then $r\preceq t$.
\end{cor}
\begin{proof}
Simple verification.
\end{proof}

\begin{thm}\label{mainlemm}Let $u,v\in E_8$ be such that $AD(u,v)\neq \emptyset$ and let $t$ be a minimal element in $AD(u,v)$. Then for all $r\in D(t)\setminus \{t\}$ we have
\[
 r\in A(u)\Leftrightarrow r \in A(v)\Leftrightarrow r^t\in D(v) \Leftrightarrow r^t\in D(u).
 \]
\end{thm}
\begin{proof}By Lemma \ref{rprect} we have $r, r^t\prec t$ and so by the minimality of $t$ we have that $r\in A(u)$ implies $r\in A(v)$ and $r^t\in D(v)$ implies $r^t\in D(u)$. Therefore we only have to prove that $r \in A(v)$ implies $r^t\in D(v)$ and that $r^t\in D(u)$ implies $r\in A(u)$. Let $x$ be any point in $\mathcal C_u$ and $y$ be any point in $\mathcal C_v$.
 We proceed case by case.

 (1) $t={s}_l$, $r=s_{h,l}$, with $h<l$ and so $r^t={s}_h$.\\
  If $s_{h,l}\in A(v)$ we have $y_h<y_l$. Since ${s}_l\in D(v)$ we have $y_l<0$ and therefore $y_h<0$, i.e. ${s}_h\in D(v)$.  Finally, if ${s}_h\in D(u)$ we have $x_h<0$ and this, together with $x_l>0$ implies $s_{h,l}\in A(u)$.

 (2) $t={s}_l$, $r={s}_J$, $|J|=3$ with $l\in J$ and so $r^t={s}_{\bar J\triangle L}$.\\
  If ${s}_J\in A(v)$, since ${s}_l\in D(v)$ we have $\frac{1}{3}\sigma(y)-\Sigma_J(y)<0$ and $y_l<0$. We have
  \begin{align*}
   \frac{2}{3}\Sigma(y)-\Sigma_{\bar J\triangle L }(y)&=\frac{2}{3}\Sigma(y)+\Sigma_J(y)-\Sigma-y_l\\
   &= -\frac{1}{3}\Sigma(y)+\Sigma_J(y)-y_l\\
   &> -\frac{1}{3}\Sigma(y)+\Sigma_J(y)\\
   &>0
  \end{align*}
  and therefore ${s}_{\bar J\triangle L}\in D(v)$. Finally, if ${s}_{\bar J\triangle L}\in D(u)$ and ${s}_l\in A(u)$ we have $\frac{2}{3} \Sigma(x)-\Sigma _{\bar J\triangle L}(x)>0$ and $x_l>0$. Therefore
  \begin{align*}
   \frac{1}{3}\Sigma(x)-\Sigma_J(x)&=\frac{1}{3}\Sigma(x)+\Sigma_{\bar J\triangle L}(x)-\Sigma(x)-x_l\\
   &= -\frac{2}{3}\Sigma(x) +\Sigma_{\bar J\triangle L}(x)-x_l\\
   &<-\frac{2}{3}\Sigma(x) +\Sigma_{\bar J\triangle L}(x)\\
   &<0.
  \end{align*}

  (3) $t={s}_K$, with $|K|=6$ and  $r=s_I$ with $I=\{i,j\},\, i<j,\, i\notin K,\, j\in K$ and so $r^t=s_{K\triangle I}$.\\
  If ${s}_{I}\in A(v)$ we have $y_i<y_j$ and since ${s}_{K}\in D(v)$ we also have $\frac{2}{3}\Sigma(y)-\Sigma_{K}(y)>0$. Since $\Sigma_{K\triangle I}(y)=\Sigma_{K}(y)-y_j+y_i<\Sigma_K(y)$ we deduce that  $\frac{2}{3}\Sigma(y)-\Sigma_{K\triangle I}(y)>\frac{2}{3}\Sigma(y)-\Sigma_{K}(y) >0$ since $s_K\in D(v)$ and so $s_{K\triangle I}\in D(v)$. Finally, if $s_{K\triangle I}\in D(u)$ we have $\frac{2}{3}\Sigma(x)-\Sigma_{K\triangle I}(x)>0$ and combining this $\frac{2}{3}\Sigma(x)-\Sigma_{K}(x)<0$ we obtain $x_j-x_i>0$, i.e. $s_{i,j\in A(u)}$.

 (4) $t={s}_K$, with $|K|=6$ and $r={s}_J$ with $|J|=3$ and $J\subset K$ and so $r^t={s}_{K\triangle J}$.\\
  If  ${s}_{J}\in A(v)$ we have $\frac{1}{3}\Sigma(y)-\Sigma_J(y)<0$ and since ${s}_K\in D(v)$ we also have $\frac{2}{3}\Sigma(y)-\Sigma_K(y)>0$. By subtracting these inequalities we obtain $\frac{1}{3}\Sigma(y)-\Sigma_K(y)+\Sigma_J(y)>0$, i.e. ${s}_{K\triangle J} \in D(v)$. Finally, if ${s}_{K\triangle J} \in D(u)$ we have $\frac{1}{3}\Sigma(x)-\Sigma_K(x)+\Sigma_J(x)>0$ and since ${s}_K\in A(u)$ we also have $\frac{2}{3}\Sigma(x)-\Sigma_K(x)<0$ and subtracting these inequalities we conclude
 \[
  \frac{1}{3}\Sigma(x)-\Sigma_J(x)<0,
 \]
 i.e. $s_J\in A(u)$.

 (5) $t=s_{i,j}$, $r=s_{i,h}$ with $i<h<j$ and $r^t=s_{h,j}$.\\
 If $s_{i,h}\in A(v)$ we have $y_{i}<y_{h}$ and since $s_{i,j}\in D(v)$ we also have $y_i>y_j$ and therefore $y_h>y_j$, i.e. $s_{h,j}\in D(v)$.

 The proof is now complete as the statement for a reflection $r$ is equivalent to the same statement for the reflection $r^t$.
\end{proof}

\begin{prop}\label{goodsE8}
 Let $u,v\in E_8$ be such that $AD(u,v)\neq \emptyset$ and let $t$ be a minimal element in $AD(u,v)$. If $t\notin S$ then there exists $s\in S$ such that
 \begin{itemize}
  \item $s\in D(t)$;
  \item $s^t$ is minimal in $AD(us,vs)$.
 \item $s^t=t^s$;
 \end{itemize}
 Moreover, $s\in D_R(u)$ if and only if $s\in D_R(v)$.
 \end{prop}
\begin{proof}
We show the existing element $s$ for any reflection $t$, leaving all the details of the proof.
 If $t=s_{i,j}$ with $j>i+1$ we can choose $s=s_{j-1,j}\in D(s_{i,j})$ (see Proposition \ref{DsI}) and so $s^t=s_{i,j-1}$.

 If $t={s}_J$ with $|J|=3$ but  $J\neq \{1,2,3\}$ let $i$ be minimal such that $i\notin J$ and $i+1\in J$. Then we can choose $s=s_{i,i+1}$ (see Proposition \ref{DtJ}) and so $s^t={s}_{J\triangle \{i, i+1\}}$.

 If $t={s}_{1,2,3,4,5,6}$ we can choose $s={s}_{1,2,3}$ (see Proposition \ref{DtK}) and so $s^t={s}_{4,5,6}$.

 If $t={s}_K$ with $K\neq \{1,2,3,4,5,6\}$ let $i$ be minimal such that $i\notin K$ and $i+1\in K$. Then we can choose $s=s_{i,i+1}$ (see Proposition \ref{DtK}) and so $s^t={s}_{K\triangle \{i,i+1\}}$.

If $t={s}_1$ we can choose $s={s}_{1,2,3}$ (see Proposition \ref{DtL}) and so $s^t={s}_{1,4,5,6,7,8}$;

If $t={s}_l$ with $l>1$ we can choose $s=s_{i-1,i}$ (see Proposition \ref{DtL}) and so $s^t={s}_{i-1}$.

\end{proof}

\section{Affine Weyl groups} \label{affine}

In this section we develope some tools in the theory of affine Weyl groups that will be needed in the study of the generalized lifting property in these groups.  An affine Weyl group $\hat W$ can be realized in the following way (see \cite[Chapter 4]{humphreysCoxeter} for more details).
Let $(W,S)$ be a finite crystallographic Coxeter {system} acting as a finite reflection group on a vector space $V$
, $T$ its set of reflections and $n:=|S|$; a corresponding affine group is the group of affine transformations generated by (affine) reflections $t^{(k)}$ through all the hyperplanes of the form $H_{t^{(k)}}=\{x\in V|\, \phi_t(x)=k\}$, where $t \in T$, $k\in \mathbb{Z}$ and $\phi_t(x)=0$ is the equation of the hyperplane fixed by $t$. We assume that the linear functions $\phi_t$ are chosen in such way that the fundamental chamber $\mathcal C_0$ of $(W,S)$ sits in the halfspace $H^+_t:=\{x:\, \phi_t(x)>0\}$ for all $t\in T$. We also let $\hat T:=\{t^{(k)}:\, t\in T,\, k\in \mathbb Z\}$ be the set of (affine) reflections of $\hat W$.

Let $\mathcal{H}:=\{H_{t^{(k)}}|\,t\in T,k\in \mathbb{Z}\}$ and define $H^+_{t^{(k)}}:=\{x\in \mathbb R^n|\, \phi_t(x)>k\}$
and $H^-_{t^{(k)}}:=\{x\in \mathbb R^n|\, \phi_t(x)<k\}$. The connected components of $\mathbb{R}^n\setminus \mathcal{H}$ are called \emph{alcoves}.
The fundamental alcove $ \mathcal A_0$ is obtained by intersecting $\mathcal C_0$ with the intersection of halfspaces $\bigcap \limits_{t\in T}H^-_{t^{(1)}}$; it results that $$ \mathcal A_0=\mathcal C_0 \cap H^-_{t_0^{(1)}},$$ where $t_0$ is the maximum in $(T,\preceq)$ (see \cite[Section 4.3]{humphreysCoxeter}).

The affine Weyl group $\hat W$ acts simply transitively on the set of alcoves and therefore we identify an element $w\in \hat W$ with the corresponding alcove $\mathcal A_w:= w^{-1}( \mathcal A_0)$.

Moreover, we associate to any pair  of alcoves $\mathcal A,\mathcal B$ an integer $$d(\mathcal A,\mathcal B):= |\{H\in \mathcal{H} : \,\mbox{$H$ separates $\mathcal A$ and $\mathcal B$}\}|.$$

A sequence of alcoves $\mathcal A_0,\ldots,\mathcal A_m$ is called a gallery of length $m$ if for all $i$ we have $d(\mathcal A_i,\mathcal A_{i+1})=1$. One has (see \cite[Theorem 4.5]{humphreysCoxeter}):
\begin{prop} \label{lunghezzaffini}
 Let $w\in \hat{W}$. Then
 \begin{enumerate}
  \item $\ell(w)= d(\mathcal A_0,\mathcal A_w)$;
  \item $\ell(w)=$ the length of the shortest gallery between $\mathcal A_0$ and $\mathcal A_w$;
  \item $t^{(k)} \in D(w)$ if and only if the hyperplane $H_{t^{(k)}}$ separates $\mathcal A_0$ from $\mathcal A_w$.
 \end{enumerate}
\end{prop}

If $t\in T$ and $k\in \mathbb{Z}$, let $\mathrm{Str}(t^{(k)})$ be the hyperstripe defined by $\mathrm{Str}(t^{(k)}):=H^+_{t^{(k-1)}} \cap H^-_{t^{(k+1)}}=\{x\in \mathbb R^n:\, k-1<\phi_t(x)<k+1\}$.

\begin{prop} \label{disuguaglianza}
Let $\mathcal A,\mathcal B$ be alcoves such that $\mathcal A,\mathcal B \not \in \mathrm{Str}(t^{(k)})$. Then
$$ |d(\mathcal A,\mathcal B)-d(\mathcal A,t^{(k)}(\mathcal B))|\geq 3.$$
\end{prop}
\begin{proof}Without loss of generality we can assume that $\phi_t(x)<k-1$ for all $x\in \mathcal A$ and  $\phi_t(y)>k+1$ for all $y\in \mathcal B$. Let $\mathcal A=\mathcal A_0,\mathcal A_1,\ldots,\mathcal A_d=\mathcal B$ be a minimal gallery joining $\mathcal A$ to $\mathcal B$ and let $i$ be the (unique) index in $[d]$ such that $\phi_t(x)<k$ for all $x\in \mathcal A_i$ and $\phi_t(y)>k$ for all $y\in \mathcal A_{i+1}$. Then $t^{(k)}(\mathcal A_i)=\mathcal A_{i+1}$ and therefore
\[
 \mathcal A=\mathcal A_0,\mathcal A_1,\ldots,\mathcal A_i,t^{(k)}(\mathcal A_{i+2}),\ldots,t^{(k)}(\mathcal A_{d})=t^{(k)}(\mathcal B)
\]
is a gallery joining $\mathcal A$ and $t^{(k)}(\mathcal B)$ of length $d-1$. Nevertheless this gallery is not minimal as it crosses the hyperplane $H_{t^{(k-1)}}$ twice, as both $\mathcal A$ and $t^{(k)}(\mathcal B)$ are contained in $H^-_{t^{(k-1)}}$ while $\mathcal A_i\in H^+_{t^{(k-1)}}$, and the result follows.
\end{proof}

The following result will be used to show that in any infinite Coxeter group there is a Bruhat interval for which the generalized lifting property does not hold.
\begin{prop} \label{chamberhyperstripes}
A chamber $\mathcal C$ cannot be covered by a finite number of hyperstripes.
\end{prop}
\begin{proof}

Is is well-known that $\mathbb R^n$ is not a union of a finite number of hyperplanes and so the same hold for $\mathcal C$. So let $x\in \mathcal C$ such that $\phi_t(x)\neq 0$ for all $t\in T$.  As all $\phi_t$ are linear it is clear that for every $a\in \mathbb R$ there exists $c>0$ such that $\phi_t(cx)>a$ for all $t\in T$. As $cx\in \mathcal C$ for all $c>0$ the result follows.
\end{proof}


\section{The generalized lifting property} \label{sectionGLP}

The lifting property stated in Proposition \ref{sollevamento} does not ensure that for a Bruhat interval $[u,v]$ there exists a simple
reflection $s\in D_R(v)\setminus D_R(u)$ such that $u\vartriangleleft us \leqslant v$ and $u\leqslant vs \vartriangleleft v$. Of course,
if $D_R(v) \subseteq D_R(u)$, such an $s$ does not exist; in general one can hope that there exists a reflection $t\in AD(u,v)$ for which
$u\vartriangleleft ut \leqslant v$ and $u\leqslant vt \vartriangleleft v$. This is called the \emph{generalized lifting property} (GLP) of the interval $[u,v]$ and \cite[Theorem 3.3]{TsukermanWilliams} asserts that in the symmetric groups $A_n$ the GLP holds for every interval.
Clearly, by Proposition \ref{sollevamento}, an interval of the type $[e,u]$ has the GLP, for every $u>e$, and, in finite groups, the same happens for intervals of the type $[u,w_0]$, for every $u<w_0$, where $w_0$ denote the unique element of maximal length.

The following example shows that this is not the case in general.

\begin{ex} \label{esempioI4}
Let $(W,S)$ be a Coxeter group of rank 2 with $S=\{s,t\}$ and $m_{s,t}\geq 4$ (for example the Weyl group of type $B_2$). Consider the elements
$u=s$ and $v=sts$. Then $\ell(u,v)=2$ and $[u,v]=\{s,st,ts,sts\}$.
Therefore $AD(u,v)=\{sts, ststs\}$ but  $\ell(uststs)=4>\ell(u)+1$ and $\ell(vsts)=0<\ell(v)-1$.
\end{ex}
The example above shows that there is an interval in any Coxeter group which is not simply-laced for which the GLP does not hold.

The following theorems affirm that every interval
in a finite simply-laced Coxeter groups, i.e. a direct product of Weyl groups of type $A_n$, $D_n$, $E_6$, $E_7$ and $E_8$, has the GLP.





\begin{thm}[Covering property]\label{theoremcovering}Let $W$ be a finite simply-laced Coxeter system, $u,v\in W$ be such that $AD(u,v)\neq \emptyset$ and $t$ be a minimal element in $AD(u,v)$. Then $u\lhd ut$ and $v\rhd vt$.
\end{thm}
\begin{proof}
If $W=W_1\times W_2$ is a direct product of two Coxeter groups the set of reflections of $W$ is given by the union of the sets of reflections of $W_1$ and $W_2$ and we deduce that we can assume that $W$ is irreducible. Morevoer, we observe that if the statement is valid for a Coxeter group $W$ then it holds also for every parabolic subgroup of $W$ and so it is sufficient to prove the result for the groups $D_n$ and $E_8$ only.

To show that $u\lhd ut$ we have to prove that $\ell(u)=\ell(ut)-1$ and for this it is enough to show that there exists a bijection $\phi:D(u)\rightarrow D(ut)\setminus\{t\}$.
   The bijection $\phi$ is the restriction to $D(u)$ of the following involution on $T$: we let
   \begin{equation*}
     \phi(r)=\begin{cases}
             r^t,&\textrm{if $r\in A(t)$,}\\r,& \textrm{otherwise},
            \end{cases}
   \end{equation*} for all $r\in T$.

   The map $\phi$ is a bijection on $T$ by Lemma \ref{lemmainvarianza} and so to conclude we have to prove that if $r\neq t$ then
   \begin{equation}\label{phibij}
   r\in D(u) \Leftrightarrow \phi(r)\in D(ut).
   \end{equation}
   First assume $r\in D(t)$, and so also $r^t\in D(t)$ by Lemma \ref{lemmainvarianza}. By Proposition \ref{crucialprop} applied to $r^t$ we have $r^t\in D(u)$ if and only if $r\in A(ut)$. By Theorems  \ref{tt*Dn} and  \ref{mainlemm} we also have $r^t\in D(u)$ if and only if $r\in A(u)$ and therefore Eq. \eqref{phibij} follows in this case.

   If $r\in A(t)$ all the statements of Proposition \ref{crucialprop} are false and in particular we have that \eqref{phibij} is implied by the negation of condition (3) in Proposition \ref{crucialprop}.

   The result for $v$ follows similarly by observing that $\phi$ restricts to a bijection
   $$ \phi:D(v)\setminus\{t\}\rightarrow D(vt).$$

\end{proof}

 We note that the covering property (Theorem \ref{theoremcovering}) does not require that $u\leq v$.

\begin{thm}[Generalized lifting property]\label{GLPfinite}
 Let $(W,S)$ be a finite simply-laced Coxeter system, $u,v\in W$, $u<v$ and $t$ be a minimal element in $AD(u,v)$. Then $u\vartriangleleft ut \leqslant v$ and $u\leqslant vt \vartriangleleft v$.
\end{thm}
\begin{proof} Once again it is sufficient to prove the result for the groups $D_n$ and $E_8$.  By Theorem \ref{theoremcovering} we only have to show that $ut\leq v$ and $u\leq vt$.
We proceed by induction on the partial order $\prec$ on reflections. If $t$ is minimal in $T$ then $t\in S$ and the result follows by Proposition \ref{sollevamento}. Otherwise choose $s\in S$ as in Propositions \ref{goodsDn} and \ref{goodsE8}. Since $s^t \in D(t)$ we have $s^t\prec t$ by Corollary \ref{twistDcor} and Corollary \ref{rprect} and we can apply our induction hypothesis to the pair $(us,vs)$. It follows that $uss^t< vs$ or, as the third points in Propositions \ref{goodsDn} and \ref{goodsE8} show, $uts<vs$. To deduce that $ut\leq v$ it is enough to prove that $s\in D(ut)$ if and only if $s\in D(v)$.
But $s\in D(t)$ and so $s\in D(ut)$ if and only if $s^t\in A(u)$ by Proposition \ref{crucialprop} and hence the result follows from Theorems \ref{tt*Dn} and \ref{mainlemm}.
The proof of $u< vt$ is similar.

\end{proof}

The next proposition generalizes the result stated in Proposition \ref{propRsimm}.
\begin{prop} Let $[u,v]$ be a Bruhat interval in a finite simply-laced Coxeter group and $t$ a minimal element in $AD(u,v)$. Then $$R_{u,v}=qR_{ut,vt}+(q-1)R_{u,vt}.$$
\end{prop}
\begin{proof} If $W=W_1\times W_2$ is a product of two Coxeter groups then $u=u_1u_2$ and $v=v_1v_2$ with $u_1,v_1\in W_1$ and $u_2,v_2\in W_2$. By \cite[Proposition 1.7]{BiBr}, we have $R_{u,v}=R_{u_1,v_1}R_{u_2,v_2}$ and since $t\in W_1\cup W_2$ we can easily deduce that it is enough to prove the statement if $W$ is irreducible; in particular we can assume that either $W=D_n$ or $W=E_8$ as the other irreducible finite simply-laced Coxeter system are parabolic subgroups of these. We prove the result by induction on the rank function of $(T,\preceq)$. If $t\in S$ the result is the recursion \eqref{ricorsioneRpolinomi}.
Otherwise, choose $s\in S$ as in Lemmas \ref{goodsDn} and \ref{goodsE8}. Then, since  $s\in D_R(v)$ if and only if $s\in D_R(u)$, we have
$R_{u,v}=R_{us,vs}=qR_{ustst,vstst}+(q-1)R_{us,vstst}=qR_{uts,vts}+(q-1)R_{us,vts}=qR_{ut,vt}+(q-1)R_{u,vt}$. The last
equality follows from the fact that $v<vs$ if and only if $vt<vts$ if and only if $u<us$ if and only if $ut<uts$.
\end{proof}

The following example shows that the GLP does not hold in general for infinite simply-laced Coxeter system.

\begin{ex}
Let $n\geqslant 3$ and consider the affine Weyl group $\tilde{A}_n$,  with set of Coxeter generators $S=\{s_1,s_2,...,s_n\}$ (see \cite[Section 8.3]{bjornerbrenti} for a combinatorial description).
Let $v=s_1s_2\cdots s_{n-1}s_ns_1s_2\cdots s_{n-1}$ and $u=s_1s_2...s_{n-1}$. In the combinarial description of $u$ and $v$ as permutations of $\mathbb Z$ we have $v=[v(1),\ldots,v(n)]=[3,4,...,n,n+2,-n+1]$ and $u=[u(1),\ldots,u(n)]=[2,3,...,n,1]$.
Letting $t_i= u^{-1}s_n\cdots s_{i+1}s_is_{i+1}\cdots s_n u$ for all $i\in [n]$ one can check that $AD(u,v)=\{t_1,t_2,\ldots,t_n\}$. 
Nevertheless, using the combinatorial interpretation of the length function in $\tilde A_n$ given in \cite[Section 8.3]{bjornerbrenti} one can conclude that none of the elements $t_1,\ldots,t_n$ satisfy the covering property for $[u,v]$ (in particular we have $\ell (ut_i)>\ell (u)+1$ for all $i=1,\ldots,n-1$ and $\ell (vt_n)<\ell (v)-1$).

\end{ex}

For an infinite simply-laced Coxeter system $(W,S)$, if $s\in S$, the following proposition shows that any interval of the type $[s,u]$ satisfies the GLP.
\begin{prop}
Let $(W,S)$ be a simply-laced Coxeter system and $s\in S$. Then the GLP holds for the interval $[s,u]$, for every $u>s$.
\end{prop}
\begin{proof}
  If $D_R(u)\neq \{s\}$ the result follows by Proposition \ref{sollevamento}. Let
$D_R(u)=\{s\}$; then, since $(W,S)$ is simply laced, we have that
$s_1...s_{k-1}s_ks$ is a reduced expression for $u$ and $s\not \in D_R(uss_k)$ because, otherwise,
$s_k\in D_R(u)$. Then the reflection $t=ss_ks$ satisfies $s\vartriangleleft st \leqslant u$ and $s\leqslant ut \vartriangleleft u$.
\end{proof}

Our last goal is to show that, in general, Bruhat intervals in infinite simply-laced Coxeter groups do not satisfy the GLP. Since every infinite Coxeter group has a parabolic subgroup isomorphic to an affine Weyl group it is enough to consider this class of groups.

The next lemma states that any element in an affine Weyl group is smaller in the Bruhat order than every element of large enough length.

\begin{lem} \label{lemmabruhat}
  Let $(W,S)$ be an affine Coxeter system. Then for any $u \in W$ there exists $n \in \mathbb{N}$ such that
  $u< v$ for all $v \in W$ whose length satisfies $\ell(v)> n$.
\end{lem}
\begin{proof} Since every parabolic subgroup $W_J$ of $W$ is finite, we let $k$ be the maximum length of an element in a proper parabolic subgroup of $W$.
If $u\in W$ and $s_1s_2...s_h$ is a reduced expression for $u$, consider the number $n=h(k+1)$.
If $v\in W$ is any element whose length satisfies $\ell(v)>n$, then $v$ should has a reduced expression such as $v_1v_2v_3...v_h$, with $\ell(v_i)>k$ for all $i$. In particular $v_i$ is not contained in any parabolic subgroups of $W$ and hence $s_i<v_i$ for all $i\in [h]$ and so $u<v$.




\end{proof}

\begin{thm} \label{GLPinfinite} Let $(\hat{W},S)$ be an affine Coxeter system
 and $\mathcal A_u$ be any alcove contained in the chamber $-\mathcal C_0$. Then there exists $v\in \hat{W}$ such that $u<v$ but the interval $[u,v]$ does not satisfy the covering property.
\end{thm}
\begin{proof} Let $W$ and $T$ be as in \S \ref{affine}.
 The fundamental alcove is given by the inequalities $0<\phi_t(x)<1$ for all $t\in T$. If $\mathcal A=\mathcal A_u$ is an alcove in $-\mathcal C_0$, then it satisfies $\mathcal A\in H^-_t$ for every $t\in T$; therefore $T\subseteq D(u)$ and $\{t^{(1)}:\,t\in T\} \subseteq A(u)$. Consider the finite set $R(u)=\{t^{(k)}\in \hat{T}:\, u\lhd ut^{(k)} \}$; by Proposition \ref{chamberhyperstripes}, we can find an alcove  $\mathcal A_v\in-\mathcal C_0$ such that

 \[
  \mathcal A_v \not \subset \bigcup_{t^{(k)}\in R(u)} \mathrm{Str}(t^{(k)}).
 \]
Moreover, such element $v$ can be chosen of arbitrary high length (as $-\mathcal C_0 \setminus \bigcup_{t^{(k)}\in R(u)} \mathrm{Str}(t^{(k)}$ is unbounded), and in particular we can assume that $u<v$ by Lemma \ref{lemmabruhat}.

 Now let $t^{(k)} \in R(u)$. It is enough to show that $v$ does not cover $vt^{(k)}$. We already know that $k\neq 0$ by construction (since $t=t^{(0)}\in D(u)$) and that $t^{(1)}\in A(v)$ by our previous remark. For $k\neq 0,1$ we have that the fundamental alcove $\mathcal A_0$ is not contained in the hyperstripe $\mathrm{Str}(t^{(k)})$ and so we can apply Proposition \ref{disuguaglianza} to the pair $\mathcal A_0$ and $\mathcal A_v$. In particular we have
 \[
  |\ell(v)-\ell(vt^{(k)})|=|d(\mathcal A_0,\mathcal A_v)-d(\mathcal A_0,t^{(k)}(\mathcal A_v))|\geq 3
 \]
and the proof is complete.

\end{proof}
By the theorem above, in every infinite Coxeter group there exists an interval which does not satisfy the GLP. Therefore with Example \ref{esempioI4}, Theorem \ref{GLPfinite} and Theorem \ref{GLPinfinite} we have proved the following theorem.

\begin{thm} \label{teoremaGLP}
  Let $(W,S)$ be a Coxeter system. The GLP holds for every Bruhat interval $[u,v]$ in $W$ if and only if
  $(W,S)$ is finite and simply-laced.
\end{thm}


\end{document}